\documentclass{amsart}
\usepackage{amssymb,amsmath,amsrefs}
\usepackage[colorlinks=true]{hyperref}
\hypersetup{urlcolor=blue, citecolor=green}
\usepackage[dvips]{graphicx}
\usepackage{color}
\usepackage{tikz}
\usetikzlibrary{shadows.blur}

\theoremstyle{plain}

\newtheorem*{conj*}{Conjecture}
\newtheorem*{cor*}{Corollary}
\newtheorem{theorem}{Theorem}[section]

\newtheorem{prop}[theorem]{Proposition}
\newtheorem{proposition}[theorem]{Proposition}
\newtheorem{corollary}[theorem]{Corollary}

\newtheorem*{q*}{Question}

\theoremstyle{definition}
\newtheorem{definition}[theorem]{Definition}
\newtheorem*{def*}{Definition}
\newtheorem{remark}[theorem]{Remark}

\newcommand{\R}{\mathbb{R}}     
\newcommand{\N}{\mathbb{N}}   
\newcommand{\h}{\mathcal{H}} 

\newcommand{\x}{\mathcal{X}}

\def\d{\mbox{diam}}

\renewcommand{\epsilon}{\varepsilon}

\newcommand{\eps}{\varepsilon}

\newcommand{\diam}{\operatorname{diam}}

\title{Stable/unstable continua of cw-expansive flows}
\author{A. Artigue, B. Carvalho, and M. Tacuri}
\date{\today}

\thanks{2020 \emph{Mathematics Subject Classification}: Primary 37B05; Secondary 37B45, 37C10.}
\keywords{Cw-expansiveness, Flows, stable/unstable sets, continua.}

\begin{document}
\begin{abstract}
We introduce distinct definitions of local stable/unstable sets for flows without fixed points, namely, kinematic, geometric, and sectionally geometric, and discuss relations between them. We prove the existence of continua with a uniform diameter within each sectionally geometric local stable/unstable set for cw-expansive flows defined on Peano continua. 
\end{abstract}

\maketitle
\vspace{-0.4cm}

\section{Introduction}

In chaotic systems, the study of distinct initial conditions with orbits that separate is essential to understanding their dynamics. Indeed, they appear in the study of several dynamical properties and invariants of such systems, such as the topological entropy, sensitivity to initial conditions, and expansiveness, among others. The study of points with orbits that, instead, do not separate in the future/past also helps in this endeavor, and the understanding of the structure of local stable/unstable sets is central in the field of dynamical systems. In hyperbolic systems, the Stable Manifold Theorem ensures that local stable/unstable sets are manifolds tangent to the stable/unstable spaces, and a big part of the hyperbolic dynamics can be obtained from this structure (see \cite{Robinson} for example). In the classification of expansive surface homeomorphisms of Hiraide/Lewowicz, local stable/unstable sets are part of a pair of transversal singular foliations with a finite number of singularities and they conclude that the homeomorphism is pseudo-Anosov (see \cite{Hiraide1} and \cite{LewowiczSurfaces}).

In generalizations of hyperbolicity and expansiveness, the structure of local stable/unstable sets can be much more complicated. This can be seen for example in Walter's pseudo-Anosov diffeomorphism of $\mathbb{S}^2$ (see \cite{Walters}), which is an example of a continuum-wise expansive homeomorphism, where local stable/unstable sets can contain a cantor set of disjoint arcs (see \cite{ArtigueDend}). Cw-expansiveness was introduced by H. Kato in \cites{Kato93,Kato93B} in the context of homeomorphisms and further explored in \cites{AAV,ArtigueDend,ACCV,ACCV2,ACCV3,CC,CC2,CR}, among others. An important result in this matter is the existence of a continuum (compact and connected) with uniform diameter inside each local stable/unstable set. In the case of continuous flows, M. Paternain proved in \cite{PM}*{Lemma 4} the non-existence of stable points for expansive flows, and as a consequence, obtained in \cite{PM}*{Lemma 5} the existence of a continuum inside each local stable/unstable set. The study of cw-expansive flows started in \cite{Cwe1} where the existence of one non-trivial local unstable continuum was obtained (in Lemma 6.4) to prove that the topological entropy of cw-expansive flows in spaces of topological dimension greater than one is positive. In this article, we generalize all these results proving that cw-expansive flows defined on a Peano continuum have a continuum within each local stable and local unstable set.



\textit{Contents of the paper}. In Section 2, we study local stable sets and some
basic notions about flows. In Section 3, we prove uniform expansivity (cw-expansivity). In
Section 4, we study cw-expansive flows from the viewpoint of fields of
local transversal sections. In Section 5, we prove our main results concerning
the existence of local stable continua and the non-existence of stable
points for flows on Peano continua.

\section{Basic Properties of Local Stable and Local Unstable Sets}
In this section, we discuss distinct definitions of local stable/unstable sets for flows, namely, kinematic, geometric, and sectionally geometric. The definitions involve reparametrizations of orbits and transversal sections, so in this section, we will state some basic definitions and results involving these notions. Let $(X, \phi)$ be a flow defined in a metric space $(X,d)$, that is a continuous map $\phi: \mathbb{R} \times X \rightarrow X$, denoted as $\phi(t, x) = \phi_{t}(x)$, satisfying
\begin{enumerate}
    \item $\phi_{0}(x) = x$ for every $x \in X$ and
    \item $\phi_{s+t}(x) = \phi_{s}\left(\phi_{t}(x)\right)$ for all $s, t \in \mathbb{R}$ and every $x \in X$. 
\end{enumerate}

\begin{definition}[Kinematic stable/unstable sets]
For each $x \in X$ and $\varepsilon > 0$, we define the \textit{kinematic $\varepsilon$-stable} and \textit{kinematic $\varepsilon$-unstable} sets of $x\in X$ as follows:
$$W^s_\varepsilon(x)=\{y \in X: d(\phi_t(x), \phi_t(y))\leq \varepsilon\, \, \, \text{for all} \, \, t \geq 0\},$$
$$W^u_\varepsilon(x)=\{y \in X:  d(\phi_t(x), \phi_t(y))\leq \varepsilon\, \, \text{for all} \, \, t\leq 0\}.$$
\end{definition}

The words `kinematic' and `geometric' first appeared in various forms of defining expansive systems. In fact, we can see that they were first described in \cite{Cerlewoxicz} and later in \cite{Kinematic}. We observe that this literature used on expansive systems can help us define stable and unstable sets in different ways. 
The kinematic local stable/unstable sets are the classical notions of local stable/unstable sets of the hyperbolic dynamics, but since we are working with non-hyperbolic flows we call them kinematic to separate it from the other definitions.

We denote by $Rep(\R)$ the set of all increasing homeomorphisms $h: \R \rightarrow \R$ such that $h(0)=0$. We also denote by $Rep(\R)^{+} \left( Rep(\R)^{-}\right)$ as the set of functions that are restrictions of functions in $Rep(\R)$ to $[0, \infty)$ (resp. $(-\infty,0]$).  In the definition of local stable/unstable geometric sets that we will define next, we allow one of the orbits to undergo changes over time, that is, one of the orbits to be reparameterized.

\begin{definition}[Geometric stable/unstable sets]\label{3.0}
For each $x\in X$ and $\varepsilon>0$, we define the \textit{geometric $\varepsilon$-stable} and \textit{geometric $\varepsilon$-unstable} sets of $x\in X$ as follows:
$$W^{ss}_\varepsilon(x)=\{y \in X: \exists \ \alpha\in Rep(\R)^+ \hspace{0.2cm}\text{such that}\hspace{0.1cm}d(\phi_t(x), \phi_{\alpha(t)}(y))\leq \varepsilon\, \, \text{for all} \, \, t \geq 0\},$$
$$W^{uu}_\varepsilon(x)=\{y \in X: \exists \ \alpha\in Rep(\R)^- \hspace{0.2cm}\text{such that }\hspace{0.1cm} d(\phi_t(x), \phi_{\alpha(t)}(y))\leq \varepsilon\, \, \text{for all} \, \, t\leq 0\}.$$ Clearly, $W^s_{\eps}(x)\subset W^{ss}_{\eps}(x)$ and $W^u_{\eps}(x)\subset W^{uu}_{\eps}(x)$ for every $x\in X$.
\end{definition}
In the definition of sectionally geometric local stable/unstable sets, we will use sections transversal to the flow. The use of transversal sections is classical in the study of regular flows\footnote{We say that $\phi$ is a regular flow if, for every $x \in X$, there exists a $t \in \mathbb{R}$ such that $\phi_{t}(x) \neq x$, meaning it has no equilibrium points.} (see \cite{TR}*{Lemma 4.1}). In \cite{TR} Thomas covers the space with a finite number of box flows of transversal sections and proves several properties of local stable/unstable sets expansive flows with canonical coordinates. More recently, Artigue \cite{Aa} constructed a field of transversal sections, meaning that he considers a transversal section at every point of the space. This field of sections facilitates the proof of some known facts of expansive flows. We explain below some basic properties of this field of sections and use them to define the sectionally geometric local stable/unstable sets. 

\begin{definition}[Field of transversal sections]
Assume that $\phi: \mathbb{R} \times X \rightarrow X$ is a regular flow and let $\mathcal{K}(X)$ be the set of all compact subsets of $X$. We say that $C \in \mathcal{K}(X)$ is a \textit{transversal section} through $x \in C$ if there exist $\tau>0$ and $\gamma>0$ such that 
$$B_{\gamma}(x) \subset \phi_{[-\tau, \tau]}(C) \quad \text{and} \quad C \cap \phi_{[-\tau, \tau]}(y) = \{y\} \quad \text{for every}\,\,\,\,\, y \in C,$$
where $B_{\gamma}(x)=\{y\in X; d(y,x)\leq\gamma\}$ is the closed ball of radius $\gamma$ centered at $x$.
We call a \textit{field of transversal sections} a function $H:X\rightarrow \mathcal{K}(X)$ such that there exist $\tau'>0$ and $r>0$ such that 
\[
B_{r}(x)\subset \phi_{[-\tau',\tau']}(H(x)) \hspace{0.5cm} \text{and} \hspace{0.5cm}  H(x)\cap \phi_{[-\tau',\tau']}(y)=\{y\}
\]
for every $x\in X$ and $y\in H(x)$. If $H$ is a field of transversal sections and $\phi:[-\tau,\tau] \times H(x)\rightarrow X$ is injective, for some $\tau>0$ and every $x\in X$, then we say that $H$ is a field of transversal sections of time $\tau$. 
\end{definition}

A field of transversal sections comes together with the box flow field and the projection flow defined as follows (for more details see \cite{Aa}*{Definitions 2.29 and 2.31}).

\begin{definition}[Flow box and projection]\label{1.04}
Let $H$ be a field of transversal sections of time $\tau$. We call by \textit{flow box field} associated with $H$ the field $F:X\rightarrow \mathcal{K}(X)$ defined by $F(x)=\phi_{[-\tau, \tau]}(H(x))$ and define the \textit{projection of the flow box} $\pi_{x}: F(x) \rightarrow H(x)$ as $\pi_{x}(y)=\phi_{t}(y) \in H(x)$ with $|t| \leq \tau$. 
\end{definition}

%


Before defining the sectionally geometric local stable/unstable sets, we state the following notation: we define the $\eps$-transversal section at the point $x$ by 
$$H_{\varepsilon}(x)=H(x)\cap B_{\varepsilon}(x).$$

\begin{definition}[Sectionally geometric stable/unstable sets]\label{3.01}
	For each $x\in X$ and $\varepsilon>0$, we define the  \textit{sectionally geometric $\varepsilon$-stable} and \textit{sectionally geometric $\varepsilon$-unstable} sets by:
	$$T^{s}_\varepsilon(x)=\{y \in X: \exists \ \alpha\in Rep(\R)^+ \hspace{0.2cm}\text{such that }\hspace{0.1cm} \phi_{\alpha(t)}(y)\in H_{\varepsilon}(\phi_t(x))\, \, \text{for all} \, \, t \geq 0\},$$
	$$T^{u}_\varepsilon(x)=\{y \in X: \exists \ \alpha\in Rep(\R)^- \hspace{0.2cm}\text{such that }\hspace{0.1cm} \phi_{\alpha(t)}(y)\in H_{\varepsilon}(\phi_t(x))\, \, \text{for all} \, \, t\leq 0\}.$$ 
\end{definition}
Roughly speaking, the orbits of points in the sectionally geometric local stable/unstable sets can be reparametrized to always belong to the local transversal section of the respective iterates of $x$. Clearly, $T^s_{\eps}(x)\subset W^{ss}_{\eps}(x)$ and $T^u_{\eps}(x)\subset W^{uu}_{\eps}(x)$ for every $x\in X$. The reverse inclusion is clearly not true since we can have points with close orbits with iterates not belonging to the same local transversal section. We prove below that the projection of the geometric local stable/unstable set of $x$ to the local section of $x$ is contained in the sectionally geometric local stable/unstable set of $x$. In the proof, we need additional properties of local transversal sections that we state in what follows.

\begin{proposition}\label{1.06}
	If $H$ is a field of transversal sections of time $\tau>0$, and $\gamma>0$ is such that $B_{\gamma}(x)\subset \phi_{[-\tau,\tau]}(H(x))$ for every $x\in X$, then for each $\eta\in(0,\tau]$ and each $r\in(0,\gamma)$, there exists  $\delta>0$ such that
	$$d(x,y)\leq \delta \,\,\,\,\,\, \text{implies} \,\,\,\,\,\, H_r(y)\subset \phi_{[-\eta,\eta]}(H(x)).$$ 
\end{proposition}
\begin{proof}
Its proof can be found in \cite{Aa}*{Lemma 2.19}.
\end{proof}

\begin{definition}[Symmetric and monotone fields of sections]
If $H$ is a field of transversal sections, we denote by $H^T$ the \textit{transpose} of $H$ defined by $$H^{T}(x)=\{y\in X: x\in H(y)\}.$$
We say that the field $H$ is a \textit{symmetric} field if $H^T(x)=H(x)$ for every $x\in X$.
A field of transversal sections $H: X\rightarrow \mathcal{K}(X)$ is \textit{monotone} if there exists an $\varepsilon>0$ such that for every $t\in (0,\varepsilon)$ we have $H(x)\cap H(\phi_t(x))= \emptyset$ for every $x\in X$.
\end{definition}

\begin{theorem}
    If $X$ is a metric space, then every regular flow $\phi_t$ on $X$ admits a symmetric and monotone field of transversal sections.
\end{theorem}
\begin{proof}
Its proof can be found in \cite{Aa}*{Theorem 2.51}.
\end{proof}
From now on we will fix a field of transversal sections $H:X\rightarrow \mathcal{K}(X) $, constants $\tau>0$ and $r>0$ such that $\phi:[-\tau, \tau] \times H(x) \rightarrow X$ is injective and $B(x,r)\subset \phi_{[-\tau, \tau]}(H(x))$ for every $x\in X$. For each $x\in X$ and $\varepsilon>0$, let $S_\varepsilon(x)=\pi_x(W_\varepsilon^{ss}(x))$ and $U_\varepsilon(x)=\pi_x(W_\varepsilon^{uu}(x))$, where the projection $\pi_{x}$ is described in Definition $\ref{1.04}$ (this notation follows the notation in \cite{TR}).
The following proposition states that if the orbits of two points $x,y \in X$ follow each other at all times (in the kinematic or geometric sense), then $\pi_x(y)$ and $x$ follow each other transversely at each time with some reparametrization.

\begin{proposition}\label{3.02}
	For each $\varepsilon>0$, there exists $r_1>0$ such that 
 $$S_{r_1}(x)\subseteq T_{\varepsilon}^{s}(x) \,\,\,\,\,\, \text{and} \,\,\,\,\,\, U_{r_1}(x)\subseteq T_{\varepsilon}^{u}(x).$$
\end{proposition}	
\begin{proof}
Let $\varepsilon\in(0,r)$ and choose $\lambda\in(0,\tau)$ such that 
$$d(\phi_\mu(x),x)\leq \frac{\varepsilon}{2} \,\,\,\,\,\, \text{whenever} \,\,\,\,\,\, |\mu|\leq \lambda \,\,\,\,\,\, \text{and} \,\,\,\,\,\, x\in X.$$ Now choose $r_1\in(0,\frac{\varepsilon}{2})$ (corresponding to $\lambda$) such that 
$$B_{r_1}(x)\subset \phi_{[-\lambda,\lambda]}(H(x)) \,\,\,\,\,\, \text{for every} \,\,\,\,\,\, x\in X.$$ Thus, if $y\in S_{r_1}(x)$, then $\phi_{s}(y)\in W_{r_1}^{ss}(x)$ for some $s$ with $|s|\leq\lambda$.
To prove that $y\in T_{\varepsilon}^{s}(x)$ we need to prove the existence of $\beta\in Rep(\R)^+$ such that 
$$\phi_{\beta(t)}(y)\in H_{\varepsilon}(\phi_t(x)) \,\,\,\,\,\, \text{for every} \,\,\,\,\,\, t\geq 0.$$ Since $\phi_s(y)\in W_{r_1}^{ss}(x)$, there exists $\alpha\in Rep(\R)^+$ such that 
$$d(\phi_{\alpha(t)+s}(y),\phi_t(x))\leq r_1 \,\,\,\,\,\, \text{for every} \,\,\,\,\,\, t\geq 0.$$ The choice of $r_1$ ensures, for each $t\geq0$, the existence of $\mu_t\in\R$ such that 
$$|\mu_t|\leq \lambda \,\,\,\,\,\, \text{and} \,\,\,\,\,\, \phi_{\mu_t}(\phi_{\alpha(t)+s}(y))\in H(\phi_t(x)).$$
Therefore, the following holds for every $t\geq0$:
\begin{align*}
	d(\phi_t(x),\phi_{\mu_t+\alpha(t)+s}(y))&\leq d(\phi_t(x),\phi_{\alpha(t)+s}(y))+d(\phi_{\alpha(t)+s}(y),\phi_{\mu_t+\alpha(t)+s}(y))\\
	&\leq r_1+\varepsilon/2\\
	&\leq \varepsilon.
\end{align*}
Defining $\beta\colon[0,+\infty)\to\R$ by $\beta(t)=\alpha(t)+\mu_t+s$, we obtain:
$$\phi_{\beta(t)}(y)\in H_{\varepsilon}(\phi_t(x)) \,\,\,\,\,\, \text{for every} \,\,\,\,\,\, t\geq 0.$$ To conclude that $y\in T_\varepsilon^s(x)$ we only have to prove that $\beta\in Rep(\R)^+$.

\vspace{+0.4cm}


\vspace{+0.4cm}
\hspace{-0.45cm}\textit{Continuity}: It is enough to show the continuity of $\mu_t$. Indeed, for each $\eta\in(0,\tau]$, according to Proposition \ref{1.06}, there exists a $\delta>0$ such that
		$$d(x,y)\leq \delta \quad\rightarrow \quad H_{\eps}(y)\subset \phi_{[-\eta/2,\eta/2]}(H(x)).$$
	Choose $\lambda_1>0$ such that
	\begin{itemize}
		\item $d(x,\phi_t(x))\leq \delta $ whenever $|t|\leq \lambda_1$, and
		\item $|t_1-t_2|\leq \lambda_1 \quad \rightarrow \quad |\alpha(t_1)-\alpha(t_2)|\leq \eta/2$.
	\end{itemize}
	Thus, if $|t_1-t_2|\leq \lambda_1$, then $d(\phi_{t_1}(x),\phi_{t_2}(x))\leq \delta$ and
	\begin{equation}\label{3.03}
		H_{\eps}(\phi_{t_1}(x))\subset \phi_{[-\eta/2,\eta/2]}(H(\phi_{t_2}(x))).
	\end{equation}
	Since $\phi_{\alpha(t_1)+\mu_{t_1}+s}(y)\in H_{\varepsilon}(\phi_{t_1}(x))$ and $\phi_{\alpha(t_2)+\mu_{t_2}+s}(y)\in H_{\varepsilon}(\phi_{t_2}(x))$, the inclusion in (\ref{3.03}) 
	ensures that
 $$|\alpha(t_1)+\mu_{t_1}- \alpha(t_2)-\mu_{t_2}|\leq \eta/2$$ from where we conclude that $|\mu_{t_1}-\mu_{t_2}|\leq \eta$, proving the continuity of $\mu_t$.\\
 
 \hspace{-0.45cm}\textit{Fixed at point $0$}:
 Just observe that $\pi_x(\phi_s(y))=\phi_{\mu_0+s}(y)=y$, which means $\mu_0=-s$. Therefore, $\beta(0)=0$. \\
 
\noindent
\textit{Increasing}. Due to the monotonicity of the transversal sections, it is immediate that
 $$t_1\leq t_2\quad \rightarrow\quad \mu_{t_1}+\alpha(t_1)<\mu_{t_2}+\alpha(t_2),$$
 hence $\beta(t_1)<\beta(t_2)$. Therefore, $\phi_{\beta(t)}(y)\in H_{\varepsilon}(\phi_t(x))$ for all $t\geq 0$, which means $y\in T_\varepsilon^s(x)$.
\end{proof}

\section{Uniform expansiveness and uniform cw-expansiveness}
In this section, we will show that expansiveness and cw-expansiveness imply uniform expansiveness and uniform cw-expansiveness in the transversal sections, respectively. Expansiveness is a crucial property in various areas, including chaos theory, since expansive systems exhibit chaotic behavior. In \cite{BW}, Bowen and Walters introduced a definition of expansive flow. This definition requires that distinct trajectories are separated, even if one of the orbits is reparametrized. A natural idea that appears is to consider possible extensions of this concept. In this direction, cw-expansiveness represents a generalization of expansiveness introduced by Kato \cite{Kato93} in the case of homeomorphisms. However, adapting cw-expansiveness to the context of flows is a subtle task, especially due to the flow direction and the possibility of a range of reparametrizations. This was done in W. Cordeiro's doctoral thesis and is exposed in \cite{Cwe1}.
The following is the definition of expansive flow according to Bowen and Walters. 

\begin{definition}[Expansiveness] 
A flow $(X,\phi)$ is \textit{expansive} if for every $\varepsilon>0$ there exists $\delta>0$, called an expansiveness constant of $\phi$, such that if
$$d(\phi_{\alpha(t)}(y),\phi_t(x))\leq \delta \,\,\,\,\,\, \text{for every} \,\,\,\,\,\, t\in\R$$
and some  $\alpha\in Rep(\R)$, then $y\in \phi_{(-\varepsilon,\varepsilon)}(x)$.
\end{definition}
According to Bowen and Walters (see \cite{BW}*{Lemma 1}), for expansive flows, there exists only a finite number of fixed points, and each one is an isolated point of X. This reduces our study of expansive flows to those without fixed points. To define cw-expansiveness according to W. Cordeiro we need to state first the following definitions and notations. For any subset $A$ of a metric space, let
$$C^0(A, \R)=\{h:A\rightarrow \mathbb{R}: h \, \, \text{is continuous} \}.$$ In particular, we let $C(\R)=C^0(\R, \R)$. For $A\subset X$, let
\begin{eqnarray*}
\h(A)=\{\alpha: A \rightarrow C(\R): \ \exists x_\alpha \in A \ \mbox{with} \ \alpha(x_\alpha)=id_\R \ \mbox{and}\\ 
\alpha(.)(t) \in C^0(A, \R) \,\,\, \mbox{for every} \ t \in \R\}
\end{eqnarray*}
A map $\alpha\in\h(A)$ associate to each $x\in A$ a map $\alpha(x)\in C(\R)$ where one of the associated maps is the identity and this association varies continuously with the point of $A$. The idea is that each point $x\in A$ will be reparametrized by a distinct reparametrization $\alpha(x)$. Precisely, for each $\alpha \in \h(A)$ and $t \in \R$, we define
$$\mathcal{X}^t_{\alpha}(A)=\{\phi_{\alpha(x)(t)}(x): x\in A\}.$$ Let $\mathcal{K}(X)$ be the set of all compact subsets of $X$ equipped with the Hausdorff metric and $\mathcal{C}(X)=\{A\in \mathcal{K}(X); A\, \, \mbox{is} \,\, \text{connected}\}$. In the case $A\in\mathcal{C}(X)$, a map in $\h(A)$ represents a continuum of reparametrizations.

\begin{definition}[Continuum-wise expansiveness] 
A flow $(X, \phi)$ is \textit{continuum-wise expansive}, or simply cw-expansive, if for every $\varepsilon > 0$, there exists $\delta > 0$, called a cw-expansiveness constant of $\phi$, such that if $A \in \mathcal{C}(X)$, $\alpha \in \h(A)$, and 
$$\d (\mathcal{X}^{t}_\alpha(A)) < \delta\quad  \text{for every}\quad t \in \mathbb{R},$$ then $A \subset \phi_{(-\varepsilon,\varepsilon)}(x_\alpha)$.
\end{definition}
According to \cite{Cwe1}*{Lemma 2.1}, for cw-expansive flows, there exists only a finite number of fixed points. As in the expansive case, this reduces our study of cw-expansive flows to those without fixed points. Thus, we can choose, as described in the previous section, a field of transversal sections $H\colon X\rightarrow \mathcal{K}(X) $ and constants $\tau>0$ and $r>0$ such that $\phi:[-\tau , \tau] \times H(x) \rightarrow X$  is injective and $B(x,r)\subset \phi_{[-\tau, \tau]}(H(x))$ for each $x\in X$. Additionally, let $\varepsilon_0$ be an expansiveness (cw-expansiveness) constant associated to $\tau$.

\begin{definition}[Sectional Flow] Let $I\subseteq \mathbb{R}$ be an interval. If there exists an $\alpha \in Rep(\mathbb{R})$ such that $\phi_{\alpha(t)}(y)\in H(\phi_t(x_\alpha))$ for every $t\in I$, then we define the \textit{sectional flow}
$$\Phi_{\alpha}^t(x,y)=(\phi_t(x), \phi_{\alpha(t)}(y)) \,\,\,\,\,\, \text{for every} \,\,\,\,\,\, t\in I.$$
Thus, every time we use the notation of the sectional flow $\Phi_{\alpha}^t(x,y)$, the hypothesis that $\phi_{\alpha(t)}(y)$ lies in the section $H(\phi_t(x))$ is implicit.
\end{definition}

The following theorem will be of great importance for the proof of uniform expansiveness and uniform cw-expansiveness in the transversal sections.

\begin{theorem}\label{1.01}
If $(h_n)_{n\in\N}$ is a sequence of continuous functions from $I$ to $\mathbb{R}$ with $h_{n}(0)=0$ for every $n\in\N$, $x_{n} \rightarrow x$, $y_{n} \rightarrow y$ when $n\to\infty$, and 
$$\phi_{h_{n}(t)}\left(y_{n}\right) \in H\left(\phi_{t}\left(x_{n}\right)\right) \,\,\,\,\,\, \text{for every} \,\,\,\,\,\, n \geq 1 \,\,\,\,\,\, \text{and every} \,\,\,\,\,\, t \in I,$$ then $h_{n}$ converges uniformly to an $h: I \rightarrow \mathbb{R}$ satisfying 
$$\phi_{h(t)}(y) \in H\left(\phi_{t}(x)\right) \,\,\,\,\,\, \text{for every} \,\,\,\,\,\, t \in I.$$
\end{theorem}
\begin{proof}
Its proof can be found in \cite{Aa}*{Lemma 3.2}.
\end{proof}

This theorem shows one of the important features that appear with a field of transversal sections, and particularly the sectional flow. In \cite{Oft}, uniform convergence of a sequence of reparametrizations is proved without the hypothesis that respective iterates are in the same section, but it is conditioned to the reparametrizations being in the set $$ Rep(\varepsilon)=\left\{g\in Rep^*(\R): \left|\frac{g(s)-g(t)}{s-t}-1\right|\leq \varepsilon\right\}$$ for some $\eps>0$, where $Rep^*(\R)=\{g\in Rep(\R):  g(\R)=\R  \}$. Theorem \ref{1.01} asks only that $(h_n)_{n\in\N}$ is a sequence of continuous functions fixing $0$. The transversal uniform expansiveness appeared in \cite{Keynes}*{Corollary 2.11}, where a finite number of transversal sections whose union of flow boxes cover the whole space is used (see \cite{Keynes}*{Lemma 2.4}). Working with a finite number of transversal sections makes the understanding of various concepts more difficult and the proofs of results more technical and complicate, especially if we aim to extend these properties to the cw-expansive case. This motivated us to redefine several concepts using a field of transversal sections. In the previous section, we have done this to the local stable/unstable sets and now we define transversal uniform expansiveness, which is equivalent to the uniform expansiveness defined in \cite{Keynes}.

\begin{definition}[Uniform expansiveness on transversal sections]
Let $(X,\phi)$ be an expansive flow and $\eps_0>0$ be an expansiveness constant associated to $\tau$. We say that a flow $(X,\phi)$ is \textit{uniformly transversely expansive} if for every $\varepsilon>0$, there exists $t_\varepsilon>0$ such that for every $y\in H(x)$ there exists $\alpha\in Rep(\mathbb{R})$ such that
    	$$d(x,y)\geq  \varepsilon \ \ \Rightarrow \sup_{|t|\leq t_\varepsilon}d\left(\Phi_{\alpha}^{t}(x,y)\right)> \varepsilon_0.$$ 
\end{definition}

\begin{proposition}\label{3.08a}
Every expansive flow is uniformly transversely expansive.
\end{proposition}
\begin{proof} Assume by contradiction that there exist $\varepsilon>0$, a positive sequence $t_i \rightarrow \infty$ as $i \rightarrow \infty$, a pair of sequences of points $(x_i)_{i\in\N },(y_i)_{i\in\N } \subset X$ with $d(x_i, y_i) \geq \varepsilon$, and reparametrizations $\alpha_i \in Rep(\mathbb{R})$ such that
$$\sup_{|t|\leq t_i}d\left(\Phi_{\alpha_i}^{t}(x_i,y_i)\right)=\sup_{|t|\leq t_i}d(\phi_t(x_i),\phi_{\alpha_i(t)}(y_i))\leq \varepsilon_0.$$ 
Choosing a subsequence, we have that $x_i \rightarrow x$, $y_i \rightarrow y$, $x \neq y$ and $y\in H_{\eps_0}(x)$. Theorem \ref{1.01} applied inductively to the given subsequences of $(\alpha_i)_{i\in\N}$ restricted to increasing intervals of the form $[-n,n]$ with $n\in\N$, ensures in the limit the existence of an increasing homeomorphism $\beta\in Rep(\R)$ with the property that for each $n\in\N$, there exists a sub-sequence $(\alpha_{i_k})_{k\in\N}$ of $(\alpha_i)_{i\in\N}$ such that
$$\alpha_{i_k}|_{[-n,n]}\underset{k\to+\infty}{\longrightarrow} \beta|_{[n,n]}.$$
Furthermore, for each $t \in \mathbb{R}$, there exists $n\in\N$ such that $t \in [-n, n]$ and, hence,
h

$$d(\phi_t(x),\phi_{\beta(t)}(y))= \lim\limits_{k\rightarrow\infty}d(\phi_t(x_{i_k}),\phi_{\alpha_{i_k}(t)}(y_{i_k}))\leq \varepsilon_0.$$ This and Theorem \ref{1.01} ensure that 
$$\phi_{\beta(t)}(y) \in H_{\varepsilon_0}(\phi_t (x)) \,\,\,\,\,\, \text{for every} \,\,\,\,\,\, t\in\R.$$
From expansiveness, it follows that $y\in \phi_{(-\tau,\tau)}(x)$, but since $y\in H_{\eps_0}(x)$ and $H_{\eps_0}(x)$ is a transversal section of time $\tau$, we have $y=x$ that is a contradiction. 
\end{proof}

In the following corollary, we prove the uniform contraction on the sectionally geometric local stable/unstable sets.

\begin{corollary}\label{3.091}
If $(X,\phi)$ is expansive, then for every $\varepsilon > 0$ and $x \in X$ we have $$\Phi_{\alpha}^t(T^{s}_{\varepsilon_0}(x))\subset B_\varepsilon(\phi_t(x)) \ \mbox{ for every } \  t \geq t_\varepsilon,$$ where $t_{\eps}$ is given as in the definition of uniform transversal expansiveness.
\end{corollary}
\begin{proof}
Let $y\in T_{\varepsilon_{0}}^s(x)$ and $\alpha\in Rep(\mathbb{R}^+)$ be such that 
$$\phi_{\alpha(t)}(y)\in H_{\varepsilon_{0}}(\phi_t (x)) \,\,\,\,\,\, \text{for every} \,\,\,\,\,\, t\geq 0.$$ Consider $t\geq t_\varepsilon$ and define $\beta(y)(r)=\alpha(r+t)-\alpha(t)$. Note that $\beta \in Rep([-t,t])$ and that
    $$\sup_{|r|\leq t_\varepsilon}d(\phi_{\beta(y)(r)}\left(\phi_{\alpha(t)}(y) \right), \phi_{t+r} (x))=\sup_{|r|\leq t_\varepsilon}d(\phi_{\alpha(r+t)}(y), \phi_{t+r} (x))\leq \varepsilon_{0}.$$ According to Lemma \ref{3.08a}, we have that $\phi_{\alpha(t)}(y) \in B_\varepsilon(\phi_t (x))$.
\end{proof}
Following the idea of the definition of  transverse uniform expansiveness, we will define the uniform transversal cw-expansiveness.

\begin{definition}[Cw-expansiveness on transversal sections]
Let $(X,\phi)$ be a cw-expansive flow and $\eps_0>0$ be a cw-expansive constant. We say that $(X,\phi)$ is \textit{uniformly transversely cw-expansive} if for every $\varepsilon>0$, there  exists $t_\varepsilon>0$ such that if $x\in X$ and $C\subset H(x)$ is a continuum such that $diam(C)\geq \varepsilon$, then there exists $\alpha\in \h(C)$ such that
    $$\sup_{|r|\leq t_\varepsilon}diam(\Phi_{\alpha}^r(C))>\varepsilon_{0}.$$
\end{definition}
The following lemma characterizes the uniform growth of continua on transversal sections. 
\begin{proposition}\label{3.10a}
    Every cw-expansive flow is uniformly transversely cw-expansive.
\end{proposition}

\begin{proof}
Its proof follows the same idea as the proof of Lemma \ref{3.08a}.
We only need to remark that the limit of a sequence of continua, with respect to the Hausdorff metric, is
also a continuum.
\end{proof}
Let $CT^{s}_\varepsilon(x)$ and $CT^{u}_\varepsilon(x)$ denote the connected components of $x$ in $T^{s}_\varepsilon(x)$ and $T^{u}_\varepsilon(x)$, respectively. These sets are called the $\eps$-stable sectionally geometric continuum and $\eps$-unstable sectionally geometric continuum of $x$. In the following corollary, we obtain a uniform contraction on the sectionally geometric local stable/unstable continua.
\begin{corollary} \label{3.09a}
    For every $\varepsilon>0$ and $x\in X$, 
    $$\Phi_{\alpha}^t(CT^s_{\varepsilon_0/2}(x))\subset B_\varepsilon(\phi_t(x)) \ \mbox{ if } \  t \geq t_\varepsilon.$$
\end{corollary}

\begin{proof}
Its proof follows the proof of Corollary $\ref{3.091}$ using Lemma $\ref{3.10a}$ instead of Proposition 3.6.
\end{proof}

\section{More about cw-expansive flows}

In this section, we prove important consequences of cw-expansiveness translating results of Kato to the case of continuous flows using the field of transversal sections. We begin this section defining the families of local stable/unstable continua
following the notations in \cites{Kato93,Kato93B} and the distinct notions of kinematic, geometric, and sectional geometric introduced in Section 2.
\begin{definition}[Kinematic local stable/unstable continua]
For each $\varepsilon>0$, let $\mathcal{C}^s_\varepsilon$ be the family of \textit{kinematic $\varepsilon$-stable continua} and $\mathcal{C}^u_\varepsilon$ be the family of \textit{kinematic $\varepsilon$-unstable continua} defined by:
$$\mathcal{C}^s_\varepsilon=\{A\in  \mathcal{C}(X); \d(\phi_t(A))\leq \varepsilon\, \, \text{for every} \, \, t \geq 0\},$$
$$\mathcal{C}^u_\varepsilon=\{A\in \mathcal{C}(X); \d(\phi_t(A))\leq \varepsilon\, \, \text{for every} \, \, t\leq 0\}.$$
\end{definition}
Given $A\in\mathcal{C}(X)$, we denote by $\mathcal{H}(A)^+$ the set of all maps $\alpha\colon A\rightarrow C([0,\infty))$, and by $\mathcal{H}(A)^-$ the set of all maps $\alpha\colon A\rightarrow C((-\infty,0])$, satisfying the same properties as belonging to $\mathcal{H}(A)$.

\begin{definition}[Geometric local stable/unstable continua]
For each $\varepsilon > 0$, let $\mathcal{C}^{ss}_\varepsilon$ be the family of \textit{geometric $\varepsilon$-stable continua} and $\mathcal{C}^{uu}_\varepsilon$ be the family of \textit{geometric $\varepsilon$-unstable continua} defined by:
    $$\mathcal{C}^{ss}_\varepsilon=\{A\in \mathcal{C}(X);\exists \ \alpha\in \mathcal{H}(A)^+\ \text{such that} \  \d(\x_{\alpha}^t(A))\leq \varepsilon\, \, \text{for every} \, \, t \geq 0\},$$
    $$\mathcal{C}^{uu}_\varepsilon=\{A\in \mathcal{C}(X); \exists \ \alpha\in \mathcal{H}(A)^-\ \text{such that} \ \d(\x_{\alpha}^t(A))\leq \varepsilon\, \, \text{for every} \, \, t\leq 0\}.$$
\end{definition}

\begin{definition}[Sectional flow acting on continua]
Let $I \subseteq \mathbb{R}$ be an interval. If there exists $\alpha \in \mathcal{H}(A)$ such that $$\mathcal{X}^t_{\alpha}(A)\subset H(\phi_t(x_\alpha)) \,\,\,\,\,\, \text{for every} \,\,\,\,\,\, t\in I,$$ then we define the sectional flow
$$\Phi_{\alpha}^t(x_{\alpha},A)=(\phi_t(x_\alpha), \mathcal{X}^
t_{\alpha}(A))  \,\,\,\,\,\, \text{for every} \,\,\,\,\,\, t\in I.$$
Thus, whenever we use the notation of the sectional flow, the implicit assumption is that the iterate $\Phi_{\alpha}^t(x_{\alpha},A)$ is contained in the section $H(\phi_t(x_\alpha))$. Often, we denote the continuum $\Phi_{\alpha}^t(x_{\alpha},A)$ simply by $\Phi_{\alpha}^t(A)$, when it is clear who $x_{\alpha}$ is.
\end{definition}
The next step is to define the families of local stable/unstable continua on the transversal sections.

\begin{definition}[Sectionally geometric local stable/unstable continua]\label{def3.04}
For any $\varepsilon>0$, let $\mathcal{T}^{s}_\varepsilon$ be the family of \textit{sectionally geometric $\varepsilon$-stable continua} and $\mathcal{T}^{u}_\varepsilon$ be the family of \textit{sectionally geometric $\varepsilon$-unstable continua} defined by:
$$\mathcal{T}^{s}_\varepsilon=\{A\in \mathcal{C}(X); \exists \ \alpha\in \mathcal{H}(A)^+ \, \, \text{such that} \ \Phi^t_{\alpha}(A)\subset H_\varepsilon(\phi_t (x_\alpha)) \, \, \text{for every}\,\, t\geq 0 \}$$
$$\mathcal{T}^{u}_\varepsilon=\{A\in \mathcal{C}(X); \exists \ \alpha\in \mathcal{H}(A)^- \ \text{such that} \  \Phi^
	t_{\alpha}(A)\subset H_\varepsilon(\phi_t (x_\alpha)) \, \, \text{for every}\,\,  t\leq 0 \}.$$
\end{definition}
The following proposition characterizes cw-expansiveness with respect to the families of sectionally geonetric local stable/unstable continua.

\begin{prop}\label{pro3.06}
If $(X, \phi_t)$ is a flow without fixed points, then $(X, \phi_t)$ is cw-expansive if, and only if, there exists $\delta > 0$ such that $$\mathcal{T}^{s}_\delta \cap \mathcal{T}^{u}_\delta = \{A \in \mathcal{C}(X); \#A = 1\}.$$
\end{prop}
\begin{proof}
 $(\rightarrow)$ Let $\varepsilon \in (0, \tau)$ and $\delta>0$ be a cw-expansive constant of $\phi$ associated to $\eps$. Suppose there exists a non-degenerate continuum $A \in \mathcal{T}^s_\delta \cap \mathcal{T}^u_\delta$. From Definition \ref{def3.04}, there exists $\alpha \in \mathcal{H}(A)$ such that 
 $$\Phi^t_{\alpha}(A) \subset H_\delta(\phi_t(x_\alpha)) \,\,\,\,\,\, \text{for every} \,\,\,\,\,\, t \in \mathbb{R}.$$ In particular, 
 $$\d(\Phi^t_\alpha(A))\leq \delta \,\,\,\,\,\, \text{for every} \,\,\,\,\,\, t \in \mathbb{R}$$ and the cw-expansiveness ensures that $A \subset \phi_{(-\varepsilon, \varepsilon)}(x_\alpha)$. Since $A$ is contained in a transversal section, it follows that $A = {x_\alpha}$.
 
 $(\leftarrow)$ For each $\varepsilon \in (0, \tau)$, let $\delta_1 > 0$ be given by Proposition \ref{3.02} and assume also that $B_{\delta_1}(x) \subset \phi_{[-\varepsilon, \varepsilon]}(H(x))$ for every $x\in X$. Suppose that $A \in \mathcal{C}(X)$, $\alpha \in \mathcal{H}(A)$, and 
 $$\d(\mathcal{X}^\alpha_t(A)) \leq \delta_1 \,\,\,\,\,\, \text{for every} \,\,\,\,\,\, t \in \mathbb{R}.$$ This means that $A \subseteq W^{ss}_{\delta_1}(x_{\alpha}) \cap W^{uu}_{\delta_1}(x_{\alpha})$ and hence, Proposition $\ref{3.02}$ ensures that $\pi_{x_{\alpha}}(A) \subseteq T^{s}_{\varepsilon}(x_{\alpha}) \cap T^{u}_{\varepsilon}(x_{\alpha})$. Since $\pi_{x_{\alpha}}(A) \in \mathcal{C}(X)$, we conclude that $\pi_{x_{\alpha}}(A) = \{x_\alpha\}$, which proves that $A \subset \phi_{[-\varepsilon,\varepsilon]}(x_\alpha)$. This proves cw-expansiveness.
\end{proof}

The following proposition states that if a continuum increases in diameter over a certain period of time, then in the subsequent periods, it cannot decrease significantly in diameter (see \cite{Kato93}*{Proposition 2.3} for the case of cw-expansive homeomorphisms).

\begin{proposition}\label{lema3.03}
    If $(X,\phi)$ is a cw-expansive flow, then for each $\varepsilon > 0$, there exists $\delta > 0$ such that if $C \subset H(x)$ is a continuum with $x \in C$, $\alpha \in \h(C)$, $\d(C) \leq \delta$, and $\d(\Phi_{\alpha}^s(x,C)) \leq \delta$ for some $s > 0$, then $\d(\Phi_{\alpha}^t(x,C)) < \varepsilon$ for all $t \in [0, s]$.
\end{proposition}
\begin{proof}
    Suppose there exists $\varepsilon\in(0,r)$ and sequences $\delta_k \rightarrow 0$ as $k \rightarrow \infty$, continua $C_k \subset H(x_k)$, maps $\alpha_k \in \h(C_k)$, and real numbers $0 \leq t_k < s_k$ such that $\d(C_k) \leq \delta_k$,  $\d(\Phi^{s_k}_{\alpha_k}(x_k,C_k))\leq \delta_k$, and $\d(\Phi^{t_k}_{\alpha_k}(x_k,C_k))> \varepsilon$.
    According to \cite[Theorem 1.26]{N}, there exists a path $c_k: [0, 1] \rightarrow \mathcal{C}(X)$ from $\{x_k\}$ to $C_k$ such that if $a \leq b$, then $c_k(a) \subset c_k(b)$. Consider the function
    	\begin{align*}F_k:[0,1]&\rightarrow [0,\infty]\\
		r&\mapsto F_k(r)=\sup\{\d(\Phi_{\alpha_k}^t(x_k,c_k(r))): t\in [0, s_k]\}.\end{align*}
If $a_k \in F_k^{-1}(\varepsilon)$ (note that $F_k^{-1}(\varepsilon) \neq \emptyset$ because of the assumptions), then $D_k = c_k(a_k) \subseteq C_k$ is such that
\begin{align*}
    \d(\Phi^t_{\alpha_k}(D_k)) &\leq \varepsilon \,\,\,\,\,\, \text{for every} \,\,\,\,\,\, t\in [0, s_k] \,\,\,\,\,\, \text{and}\\
    \d(\Phi^{r_k}_{\alpha_k}(D_k)) &= \varepsilon \,\,\,\,\,\, \text{for some} \,\,\,\,\,\, r_k \leq s_k,
\end{align*}
see Figure \ref{fig1a}.
  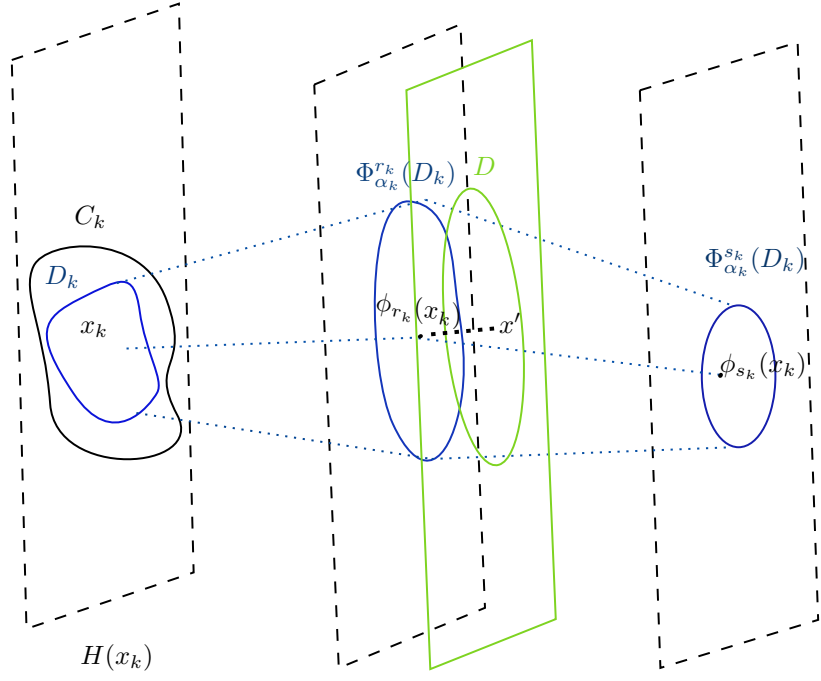
\begin{figure}[ ht!]
		\centering
		\label{fig1a}
\tikzset{every picture/.style={line width=0.75pt}} 
\begin{tikzpicture}[x=0.75pt,y=0.75pt,yscale=-1,xscale=1]
	\path (0,345);\path (492.77593994140625,0);
	\draw  [color={rgb, 255:red, 0; green, 0; blue, 0 }  ,draw opacity=1 ][fill={rgb, 255:red, 255; green, 255; blue, 255 }  ,fill opacity=1 ][dash pattern={on 4.5pt off 4.5pt}] (52.2,29.45) -- (136.5,1) -- (143.93,288.08) -- (59.63,316.53) -- cycle ;
	\draw  [color={rgb, 255:red, 0; green, 0; blue, 0 }  ,draw opacity=1 ][fill={rgb, 255:red, 255; green, 255; blue, 255 }  ,fill opacity=1 ][dash pattern={on 4.5pt off 4.5pt}] (204.76,41.85) -- (278.5,11.46) -- (290.87,305.95) -- (217.13,336.34) -- cycle ;
	\draw   (63.27,133.35) .. controls (69.83,120.14) and (117.41,115.35) .. (130.27,150.63) .. controls (143.13,185.91) and (120.65,180.48) .. (135.67,208.66) .. controls (150.7,236.83) and (74.46,241.77) .. (71.58,203.19) .. controls (68.69,164.62) and (56.71,146.56) .. (63.27,133.35) -- cycle ;
	\draw  [color={rgb, 255:red, 12; green, 20; blue, 217 }  ,draw opacity=1 ][fill={rgb, 255:red, 255; green, 255; blue, 255 }  ,fill opacity=1 ] (86.77,149.85) .. controls (106.77,139.85) and (113.82,135.19) .. (116.72,156.2) .. controls (119.62,177.22) and (129.15,190.76) .. (125.53,197.34) .. controls (121.91,203.92) and (103.23,227.56) .. (83.23,197.56) .. controls (63.23,167.56) and (66.77,159.85) .. (86.77,149.85) -- cycle ;
	\draw  [color={rgb, 255:red, 20; green, 54; blue, 186 }  ,draw opacity=0.88 ][fill={rgb, 255:red, 255; green, 255; blue, 255 }  ,fill opacity=1 ] (252.31,100.82) .. controls (274.55,102.14) and (274.64,133.74) .. (276.51,148.25) .. controls (278.37,162.76) and (283.61,185.98) .. (276.76,211.01) .. controls (269.9,236.03) and (256.89,240.12) .. (244.45,213.2) .. controls (232.02,186.27) and (230.07,99.51) .. (252.31,100.82) -- cycle ;
	\draw  [color={rgb, 255:red, 0; green, 0; blue, 0 }  ,draw opacity=1 ][fill={rgb, 255:red, 255; green, 255; blue, 255 }  ,fill opacity=1 ][dash pattern={on 4.5pt off 4.5pt}] (368.97,44.75) -- (448.54,20.51) -- (458.9,312.02) -- (379.34,336.27) -- cycle ;
	\draw  [color={rgb, 255:red, 21; green, 30; blue, 173 }  ,draw opacity=1 ][fill={rgb, 255:red, 255; green, 255; blue, 255 }  ,fill opacity=1 ] (400.22,189.09) .. controls (400.22,169.32) and (408.53,153.29) .. (418.78,153.29) .. controls (429.03,153.29) and (437.34,169.32) .. (437.34,189.09) .. controls (437.34,208.86) and (429.03,224.89) .. (418.78,224.89) .. controls (408.53,224.89) and (400.22,208.86) .. (400.22,189.09) -- cycle ;
	\draw [color={rgb, 255:red, 6; green, 79; blue, 164 }  ,draw opacity=1 ] [dash pattern={on 0.84pt off 2.51pt}]  (105.21,142.06) -- (260.95,99.99) ;
	\draw [color={rgb, 255:red, 3; green, 70; blue, 149 }  ,draw opacity=1 ] [dash pattern={on 0.84pt off 2.51pt}]  (115.07,207.61) -- (263.96,230.81) ;
	\draw [color={rgb, 255:red, 7; green, 80; blue, 166 }  ,draw opacity=1 ] [dash pattern={on 0.84pt off 2.51pt}]  (260.95,99.99) -- (415.52,153.29) ;
	\draw [color={rgb, 255:red, 7; green, 80; blue, 165 }  ,draw opacity=1 ] [dash pattern={on 0.84pt off 2.51pt}]  (263.96,230.81) -- (415.52,224.89) ;
	\draw  [color={rgb, 255:red, 126; green, 211; blue, 33 }  ,draw opacity=1 ]
    (251.23,44.7) -- (314.52,19.47) -- (326.56,311.64) -- (263.28,336.88) -- cycle ;
	\draw  [color={rgb, 255:red, 126; green, 211; blue, 33 }  ,draw opacity=1 ] (271.18,169.74) .. controls (266.89,131.3) and (271.88,97.65) .. (282.33,94.57) .. controls (292.77,91.49) and (304.71,120.14) .. (308.99,158.58) .. controls (313.27,197.01) and (308.28,230.66) .. (297.84,233.75) .. controls (287.4,236.83) and (275.46,208.17) .. (271.18,169.74) -- cycle ;
	\draw [color={rgb, 255:red, 5; green, 81; blue, 170 }  ,draw opacity=1 ] [dash pattern={on 0.84pt off 2.51pt}]  (110.08,175.07) -- (254.99,169.61) -- (411.28,188.55) ;
	\draw [color={rgb, 255:red, 0; green, 0; blue, 0 }  ,draw opacity=1 ][line width=1.5]  [dash pattern={on 1.69pt off 2.76pt}]  (263.68,167.89) -- (298.69,164.66) ;
	\draw  [color={rgb, 255:red, 0; green, 0; blue, 0 }  ][line width=1.5] [line join = round][line cap = round] (409.43,188.85) .. controls (409.6,189.02) and (409.81,188.52) .. (409.88,188.29) ;
	\draw  [color={rgb, 255:red, 0; green, 0; blue, 0 }  ][line width=1.5] [line join = round][line cap = round] (258,169.06) .. controls (258.18,168.97) and (258.37,168.88) .. (258.56,168.78) ;
	\draw (82.37,101.79) node [anchor=north west][inner sep=0.75pt]    {$C_{k}$};
	\draw (67.43,131.85) node [anchor=north west][inner sep=0.75pt]  [color={rgb, 255:red, 33; green, 42; blue, 88 }  ,opacity=1 ]  {$\textcolor[rgb]{0.03,0.23,0.47}{D_{k}}$};
	\draw (224,78.8) node [anchor=north west][inner sep=0.75pt]  [color={rgb, 255:red, 16; green, 37; blue, 62 }  ,opacity=1 ]  {$\textcolor[rgb]{0.04,0.25,0.49}{\Phi _{\alpha _{k}}^{r_{k}} (D_{k})}$};
	\draw (283.47,76.65) node [anchor=north west][inner sep=0.75pt]  [color={rgb, 255:red, 65; green, 117; blue, 5 }  ,opacity=1 ]  {$\textcolor[rgb]{0.49,0.83,0.13}{D}$};
	\draw (400.12,122.26) node [anchor=north west][inner sep=0.75pt]    {$\textcolor[rgb]{0.03,0.21,0.41}{\Phi _{\alpha _{k}}^{s_{k}} (D_{k})}$};
	\draw (85.16,160.02) node [anchor=north west][inner sep=0.75pt]    {$x_{k}$};
	\draw (234.29,143.58) node [anchor=north west][inner sep=0.75pt]  [rotate=-9.1,xslant=-0.17]  {$\phi _{ r_{k}}(x_{k})$};
	\draw (296.57,155.46) node [anchor=north west][inner sep=0.75pt]    {$x'$};
	\draw (407.74,175.6) node [anchor=north west][inner sep=0.75pt]  [rotate=-359.61]  {$\phi _{ s_{k}} (x_{k})$};
	\draw (85.44,322.9) node [anchor=north west][inner sep=0.75pt]    {$H( x_{k})$};
\end{tikzpicture}\\
\caption{$D=\lim\limits_{k\rightarrow \infty}\Phi_{\alpha_{k}}^{r_k}(x_k,D_k)$}
	\end{figure}
Continuity of $\phi$ ensures that
\begin{equation}\label{eq2.1}
		\lim_{k\rightarrow \infty}r_k=\lim_{k\rightarrow \infty}s_k-r_k=\infty.
	\end{equation}
Indeed, given $s>0$, continuity ensures the existence of $\delta_s > 0$ such that $d(x, y) < \delta_s$ imples $d(\phi_t(x), \phi_t(y)) \leq \frac{\varepsilon}{2}$ whenever $|t| \leq s$. If $\delta_{k}<\delta_s$, then
$$\d (C_k)<\delta_s \ \  \text{and} \ \  \d (\Phi_{\alpha_k}^{s_k}(C_k))<\delta_s,$$
and consequently 
$$\d(\Phi_{\alpha_k}^t(C_k))\leq \frac{\varepsilon}{2} \ \ \text{and} \ \ \d(\Phi_{\alpha_k}^{s_k-t}(C_k))\leq \frac{\varepsilon}{2}$$
for every $t \in [0, s]$. Since $\lim_{k\to+\infty}\delta_k=0$ and $\d(\Phi^{r_k}_{\alpha_k}(D_k))=\varepsilon$, it follows that $r_k > s$ and $s_k - r_k > s$ if $k$ is sufficiently large, and since $s$ is arbitrary, the statement is proved. Continuing the proof, let $D$ be an accumulation continuum of $(\Phi^{r_k}_{\alpha_k}(D_k))_{k\in\N}$ and note that $\diam(D)=\eps$ and $D\subset H(x')$, where $x'=\lim_{k\to+\infty}\phi_{r_k}(x_k)$.
As in the proof of Proposition \ref{3.08a}, we will prove the existence of $\beta \in \h(D)$ such that
    $$\d (\Phi_{\beta}^t(x',D))\leq \varepsilon \,\,\,\,\,\, \text{for every} \,\,\,\,\,\, t \in \mathbb{R}.$$
Since $D\subset H(x')$, it cannot be contained in a segment of orbit, so this contradicts cw-expansiveness.
To prove the existence of such $\beta$, let for each $k\in\N$:
    \begin{align*}
		\beta_k:\Phi^{r_k}_{\alpha_k}(x_k,D_k)&\rightarrow C^0( [-r_k,r_k]),\hspace{1.4cm}\\
		z_k &\mapsto  \beta_k(z_k)\colon [-r_k,r_k]\to \R\\
			&\phantom{\mapsto \beta_k(y_k)\colon [-r_k,r_k]} t  \to \beta_{k}(z_k)(t)=\alpha_k(y_k)(t+r_k)-\alpha_k(y_k)(r_k),
    \end{align*}
where $z_k = \phi_{\alpha_k(y_k)(r_k)}(y_k)$. For each $p \in D$, there exists a sequence $p_k \rightarrow p$ as $k \rightarrow \infty$ such that $p_k = \phi_{\alpha(y_k)(r_k)}(y_k) \in \Phi_{\alpha_{k}}^{r_k}(x_k,D_k)$ with $y_k\in D_k$. It follows from the definition of $\beta_k$ that $\phi_{\beta_k(p_k)(t)}(p_k)=\phi_{\alpha_{k}(y_k)(t+r_k)}(y_k)$ and, consequently,
$$\phi_{\beta_k(p_k)(t)}(p_k) \in H_{\varepsilon}(\phi_{t+r_k}(x_k)) \,\,\,\,\,\, \text{for every} \,\,\,\,\,\, k \geq 0 \,\,\,\,\,\, \text{and} \,\,\,\,\,\, t \in [-r_k, r_k].$$
For each $t\in\R$, there exists $r_{k_0}>0$ such that $t \in [-r_{k_0}, r_{k_0}]$
and according to Theorem \ref{1.01}, there exists a subsequence of $(\beta_k(p_k)|_{[-r_{k_0},r_{k_0}]})_{k\geq k_0} \subset C^0([-r_{k_0},r_{k_0}],\mathbb{R})$ and a $\beta(p) \in C^0([-r_{k_0},r_{k_0}])$ such that
    $$\beta_{k}(p_k)|_{[-r_{k_0},r_{k_0}]}\rightarrow \beta(p)|_{[-r_{k_0},r_{k_0}]}$$
and $\phi_{\beta(p)(t)}(p) \in H(\phi_t (x'))$. Therefore, we can ensure the existence of a continuum of reparameterizations $\beta \in \h(D)$ such that
    \begin{align*}
		\d (\Phi_{\beta}^t(D))&= \lim_{k\rightarrow \infty }\d (\Phi_{\beta_k}^t(\Phi_{\alpha_k}^{r_k}(D_k)))\\
		&=\lim_{k\rightarrow \infty }\d (\Phi_{\alpha_k}^{t+r_k}(D_k))\\
		&\leq \varepsilon.
	\end{align*}
The last inequality follows from the following fact: From (\ref{eq2.1}), for each $t\in\R$, there exists $k_0>0$ such that $0<t+r_{k}< s_{k} $ for every $k\geq k_0$. Thus, for every $k\geq k_0$, we have $\d (\Phi^{t+r_k}_{\alpha_k}(D_k))\leq \varepsilon$ and, consequently, $\lim_{k\rightarrow \infty }\d (\Phi_{\alpha_k}^{t+r_k}(D_k))\leq \varepsilon$. 
\end{proof}

\begin{remark}
In the proof, the sectional flow is important to ensure that $D$ is not contained in a segment of orbit. This is, indeed, one of the reasons that we consider the use of transversal sections to study local stable/unstable continua of cw-expansive flows.
\end{remark}

\begin{corollary}
Let $(X,\phi)$, $\varepsilon$, and $\delta$ be as in the proposition \ref{lema3.03}. Then for any $t>0$, there exists a real number $k=k(t)>0$ such that if $A\subset X$ is a continuum with $x\in A$ and $\d (A)\geq t$, then $\d (\Phi^s_\alpha (A))\geq \delta$ or $\d (\Phi^{-s}_\alpha (A))\geq \delta$ for all $s\geq k$ and all $\alpha\in \mathcal{H}(A)$.
\end{corollary}

In the next proposition, we prove that (sectionally geometric) local stable continua are stable (see \cite[Proposition 2.1]{Kato93} for the case of cw-expansive homeomorphisms).

\begin{proposition}\label{3.09}
If $(X,\phi)$ is a cw-expansive flow with a cw-expansive constant $\eps>0$ and $C\in\mathcal{T}^{s}_\varepsilon$, 
then $\lim\limits_{t\rightarrow\infty}\d (\Phi_{\alpha}^t(C))=0$.
\end{proposition}
\begin{proof}
If $C\in\mathcal{T}^{s}_\varepsilon$, then $\d (\Phi_{\alpha}^t(C))\leq \varepsilon$ for every $t\geq 0$. Suppose, by contradition, that there exists $\delta>0$ and a sequence $t_k\rightarrow\infty$ as $k\rightarrow \infty$ such that 
$$\d (\Phi^{t_k}_{\alpha}(C))\geq \delta \,\,\,\,\,\, \text{for every} \,\,\,\,\,\, k\in\N.$$ Assume by taking a subsequence if necessary, that there exists $D=\lim\limits_{k\rightarrow\infty}\Phi^{t_k}_{\alpha}(C)$ and note that $\delta\leq\d (D)\leq \varepsilon$. We will prove the existence of $\beta \in \mathcal{H}(D)$ such that 
$$\d(\Phi^t_{\beta}(D))\leq \varepsilon \quad\text{for every}\quad t\in \mathbb{R}.$$ In fact, each $y\in D$ is the limit of $y'_k=\phi_{\alpha(y_k)(t_k)}(y_k)$ with $y_k\in C$. For each $k\in\N$, define
    	\begin{align*}
		\beta_k(y'_k): [-t_k,t_k]&\rightarrow \R\\
		t&\mapsto \beta_k(y_k')(t)=\alpha(y_k)(t+t_k)-\alpha(y_k)(t_k).
  \end{align*}
and note that $\beta_k(y_k')\in \text{Rep}([-t_k,t_k])$ and $\phi_{\beta_k(y_k')(t)}(y_k')=\phi_{\alpha(y_k)(t+t_k)}(y_k)$. As in the proof of Proposition \ref{lema3.03}, we can use Theorem \ref{1.01}
to ensure the existence of $\beta\in \mathcal{H}(D)$ such that
  	\begin{align*}
		\d (\Phi_{\beta}^t(D))&= \lim_{k\rightarrow \infty }\d (\Phi_{\beta_k}^t(\phi_{\alpha}^{t_k}(D_k)))\\
		&=\lim_{k\rightarrow \infty }\d (\Phi_{\alpha}^{t+t_k}(D_k))\\
		&\leq \varepsilon.
	\end{align*}
Thus, our assertion is verified and we conclude that $D\in \mathcal{T}^s_\varepsilon\cap \mathcal{T}_{\varepsilon}^u$, which contradicts Proposition \ref{pro3.06}.
\end{proof}

\section{Existence of stable and unstable continua at each point}


In this section, we will prove the existence of continua contained within each sectionally geometric stable/unstable sets $T^s_{\eps}(x)$ and $T^u_{\eps}(x)$ for cw-expansive flows $(X,\phi)$ defined on a Peano continuum. In the case the space $X$ is a Peano continuum, additional properties on the field of transversal sections were obtained in \cite{Aa} and we state them as follows.
\begin{theorem}[\cite{Aa}*{Theorem 2.53}]\label{1.02}
    If $(X, \phi)$ is a regular flow defined on a Peano continuum $X$, then there exists a symmetric monotone field of transversal sections $H^{\prime}$ and a continuous one-parameter family of fields $H: [0, r_0] \times X \rightarrow \mathcal{C}(X)$ satisfying:
\begin{enumerate}
\item $H_{\varepsilon}: X \rightarrow \mathcal{C}(X)$ is a locally symmetric, monotone, and continuous field of transversal sections for every $\varepsilon \in (0, r_0]$,
\item $H_{\varepsilon} \subset H^{\prime}$ for every $\varepsilon \in [0, r_0]$,
\item $H_{0}(x) = {x}$ for every $x \in X$,
\item If $0 \leq \varepsilon \leq \varepsilon^{\prime} \leq r_0$, then $H_{\varepsilon} \subset H_{\varepsilon^{\prime}}$.
\end{enumerate}
\end{theorem}
The following definition contains the definition of stable point and asymptotically stable point using the field of transversal sections, as described in \cite[Definition 3.6]{Aa}.
\begin{definition}[Stable and asymptotically stable point]
A point $x\in X$ is \textit{stable} if for each $\varepsilon>0$, there exists $\delta>0$ such that $H_\delta(x)\subset T^{s}_\varepsilon(x)$. We say that $x\in X$ is \textit{asymptotically stable} if for every $\varepsilon>0$, there exists $\delta>0$ such that if $y\in H_\delta(x)$ then there exists $\alpha\in Rep(\mathbb{R})$ such that $\phi_{\alpha(t)}(y)\in H_{\varepsilon}(\phi_t(x))$ for every $t\geq 0$ and $d(\phi_{\alpha(t)}(y),\phi_t(x))\rightarrow 0$ as $t\rightarrow +\infty$. 
\end{definition}
In the following propositions, we adapt techniques of Rodríguez-Hertz in \cite{JH} to the context of cw-expansive flows using the field of transversal sections.

\begin{proposition}\label{stablesection}
Let $(X,\phi)$ be a cw-expansive flow on a Peano continuum. If $x\in X$ is a stable point, then there exists $\delta>0$ such that any $y\in H_{\delta}(x)$ is also a stable point.
\end{proposition}

\begin{proof}
Let $\eps\in(0,\frac{\eps_0}{2})$, where $\eps_0$ is a cw-expansive constant associated to $\tau$. Theorem $\ref{1.02}$ ensures the existence of $\delta\in(0,r_0)$ such that $H_\delta(y)$ is a continuum for every $y \in X$. This and the stability of $x$ ensure that $H_\delta(x) \subset CT_\varepsilon^s(x)$ (reducing $\delta$ if necessary) and since $CT_\varepsilon^s(x)\in\mathcal{T}^s_{\eps}$, Proposition \ref{3.09} ensures the existence of $\alpha \in \h_{\varepsilon}(H_\delta(x))$ such that
\begin{equation}\label{eq5a.1}
    \lim _{s \rightarrow \infty}\diam(\Phi_{\alpha}^s(H_\delta(x)))=0.
\end{equation}
Let $y \in H_{\delta}(x)$, $\varepsilon > 0$, and $\varepsilon_1 > 0$ obtained from Proposition $\ref{3.02}$. According to $(\ref{eq5a.1})$, there exists a sufficiently small $\delta_1$ such that
$$\d(\Phi_{\alpha}^t(H_{\delta_1}(x))) \leq \frac{\varepsilon_1}{3} \,\,\,\,\,\, \text{for every} \,\,\,\,\,\, t \geq 0.$$
Choose $\lambda>0$, given by uniform continuity of $\phi$, such that
$$d(z,\phi_{\mu}(z))\leq \frac{\varepsilon_1}{3} \,\,\,\,\,\, \text{whenever} \,\,\,\,\,\, |\mu|\leq \lambda \,\,\,\,\,\, \text{and} \,\,\,\,\,\, z\in X,$$
$r_1\in(0,\delta_1)$ satisfying $B(x,r_1)\subset \phi_{[-\lambda, \lambda]}(H_{\delta}(x))$ for all $x\in X$, and $r'\in(0,r_1)$
such that $H_{r'}(y)\subset B(x,r_1)$. Thus, we can choose a map $\mu\colon H_{r'}(y)\to [-\lambda,\lambda]$ such that $$\phi_{\mu(z)}(z)\in H_{\delta}( x) \,\,\,\,\,\, \text{for every} \,\,\,\,\,\, z\in H_{r'}(y).$$
Thus, defining $\beta\in\h_{\varepsilon}(H_{r'}(y))$ as $\beta(a)(t)=\alpha(\phi_{\mu(a)}(a))(t)$ for every $a\in H_{r'}(y)$, it follows that
\begin{align*}
    \d(\mathcal{X}^{t}_{\beta}(H_{r'}(y)))&=\sup_{a,b\in H_{r'}(y)}d(\phi_{\beta(a)(t)}(a),\phi_{\beta(b)(t)}(b) )\\
    &=\sup_{a,b\in H_{r'}(y)}d(\phi_{\alpha(\phi_{\mu(a)}(a))(t)}(a),\phi_{\alpha(\phi_{\mu(b)}(b))(t)}(b) )\\
   &\leq \sup_{a\in H_{r'}(y)}d(\phi_{\alpha(\phi_{\mu(a)}(a))(t)}(a),\phi_{\alpha(\phi_{\mu(a)}(a))(t)+\mu(a)}(a) )\\
    &+\sup_{a,b\in H_{r'}(y)}d(\phi_{\alpha(\phi_{\mu(a)}(a))(t)+\mu(a)}(a), \phi_{\alpha(\phi_{\mu(b)}(b))(t)+\mu(b)}(b))\\
    &+\sup_{b\in H_{r'}(y)}d(\phi_{\alpha(\phi_{\mu(b)}(b))(t)+\mu(b)}(b),\phi_{\alpha(\phi_{\mu(b)}(b))(t)}(b)) )\\   &\leq \frac{\varepsilon_1}{3}+\frac{\varepsilon_1}{3}+\frac{\varepsilon_1}{3}\\
    &= \varepsilon_1,
\end{align*}
for every $t\geq 0$. This does not necessarily prove that $H_{r'}(y) \subset W^{ss}_{\varepsilon_1}(y)$. Indeed, each $a \in H_{r'}(y)$ has $\beta(a)$ as its respective reparametrization, but we have to ensure that $y$ is associated to the identity, so we define
$\beta_1(.)=\beta(.)\circ\beta^{-1}(y)\in \h_{\varepsilon}(H_{r'}(y))$, which ensures that $\beta_1(y)=Id$ and also maintain the inequality
$$ \d(\mathcal{X}_{\beta_1}^t(H_{r'}(y)))\leq \varepsilon_1 \,\,\,\,\,\, \text{for every} \,\,\,\,\,\, t\geq 0.$$
Indeed, note that for each $t\geq 0$ there exists $s\geq 0$ such that $t=\beta(y)(s)$, so
\begin{align*}
   \d(\mathcal{X}_{\beta_1}^t(H_{r'}(y)))&= \sup_{a,b\in H_{r'}(y)}d(\phi_{\beta_1(a)(t)}(a), \phi_{\beta_1(b)(t)}(b))\\
   &=\sup_{a,b\in H_{r'}(y)}d(\phi_{\beta(a)\circ \beta^{-1}(y)(t)}(a),\phi_{ \beta(b)\circ \beta^{-1}(y)(t)}(b))\\
    &=\sup_{a,b\in H_{r'}(y)}d(\phi_{\beta(a)\circ \beta^{-1}(y)(\beta(y)(s))}(a),\phi_{ \beta(b)\circ \beta^{-1}(y)(\beta(y)(s))}(b))\\
    &=\sup_{a,b\in H_{r'}(y)}d(\phi_{\beta(a)(s)}(a),\phi_{ \beta(b)(s)}(b))\\
    &\leq \varepsilon_1.
\end{align*}
Finally, according to Proposition 
\ref{3.02}, there exists $\beta_2\in \h(H_{r'}(y))$ such that $$ \d(\Phi_{\beta_2}^t(H_{r'}(y))\leq \varepsilon\quad \text{for all} \quad t\geq 0.$$
This ensures that $H_{r'}(y)\subset T^s_{\eps}(y)$ and that $y$ is stable.
\end{proof}

\begin{definition}[Recurrent point]
A point $x\in X$ is called recurrent if it belongs to its $\omega$-limit set $\omega(x)$, that is, there exists an increasing sequence $(t_i)_{i\in\N}$ of positive numbers such that $t_i\to\infty$ and $\phi_{t_i}(x)\to x$ when $i\to+\infty$. In particular, for each $\delta>0$ and $t_0>0$, there exists $t>t_0$ such that $\phi_t(x)\in H_\delta(x)$. 
\end{definition}
Recall that $\tau>0$ and $r>0$ are such that $\phi:[-\tau , \tau] \times H(x) \rightarrow X$ is injective and $B(x,r)\subset \phi_{[-\tau, \tau]}H(x)$ for each $x\in X$.
\begin{proposition}\label{reccurentperiodic}
If $(X, \phi)$ is a cw-expansive flow defined on a Peano continuum, then every recurrent and stable point is periodic.
\end{proposition}
\begin{proof}
Let $x \in X$ be a recurrent and stable point and $\varepsilon>0$ be a cw-expansiveness constant of $\phi$ associated to $\tau$. Theorem $\ref{1.02}$ and Proposition \ref{stablesection} ensure the existence of $\delta\in(0,r_0)$ such that $H_\delta(y)$ is a continuum for every $y \in X$ and that any $y\in H_{\delta}(x)$ is also a stable point. This and the stability of $x$ ensure that $H_\delta(x) \subset CT_\varepsilon^s(x)$ (reducing $\delta$ if necessary) and since $CT_\varepsilon^s(x)\in\mathcal{T}^s_{\eps}$, Proposition \ref{3.09} ensures the existence of $\alpha \in \h_{\varepsilon}(H_\delta(x))$ such that
\begin{equation}\label{eq5.1}
    \lim _{s \rightarrow \infty}\diam(\Phi_{\alpha}^s(H_\delta(x)))=0.
\end{equation}
Choose $r'\in(0,\delta)$ such that $B(x,r')\subset \phi_{[-\tau, \tau]}H_\delta(x)$ for each $x\in X$. Let $t_0>0$ be such that $$\diam(\Phi_{\alpha}^{t}(H_\delta(x)))<\frac{r'}{2} \,\,\,\,\,\, \text{for every} \,\,\,\,\,\, t\geq t_0$$ and use the recurrence of $x$ to ensure the existence of $t_1>t_0$ such that 
$$\phi_{t_1}(x) \in H_{\frac{r'}{2}}(x) \,\,\,\,\,\, \text{and} \,\,\,\,\,\, \Phi^{t_1}_\alpha (H_\delta(x))\subset H_{\frac{r'}{2}}(\phi_{t_1}(x)).$$ 
Thus, $\Phi^{t_1}_\alpha(H_\delta(x))\subset B(x,r')$ and we can project it into $H_{\delta}(x)$. Consider the map $\mu\colon H_\delta(x)\to\R^+$ such that
$$\Phi^{\mu(y)}_{\alpha}(y)\in H_{\delta}(x) \,\,\,\,\,\, \text{and} \,\,\,\,\,\, \alpha(y)(\mu(y))\in[\alpha(y)(t_1)-\tau,\alpha(y)(t_1)+\tau]$$
for every $y\in H_{\delta}(x)$.
In what follows, $\Phi^{\mu}_{\alpha}(H_\delta(x))$ denotes the set $$\{\Phi^{\mu(y)}_{\alpha}(y); \,\,y\in H_{\delta}(x)\}.$$ Thus, we have
$$\Phi^{\mu}_{\alpha}(H_\delta(x))\subset H_{\delta}(x)$$ 
and $\Phi^{\mu}_{\alpha}$ is a map from $H_\delta(x)$ to itself, so we can consider its iterates $(\Phi^{\mu}_{\alpha})^n$ with $n\in\N$ and obtain
$$\bigcap_{n\geq 0} (\Phi^{\mu}_{\alpha})^n(H_{\delta}(x))=\{y\}\in H_{\delta}(x).$$
(see Figure \ref{fig2.2a}). 
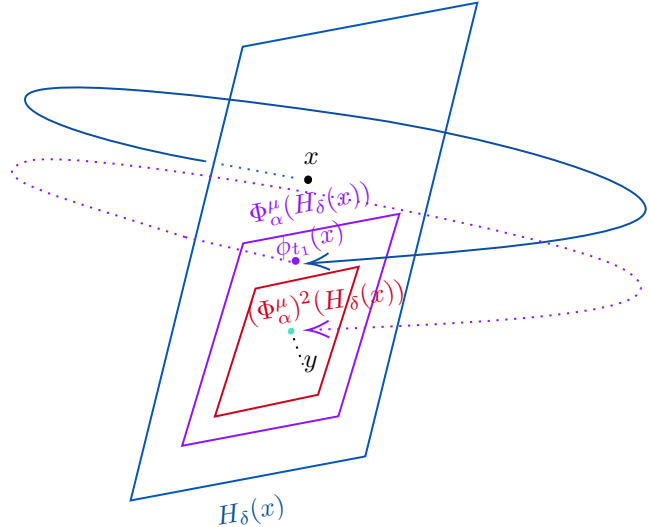
\begin{figure}[ht!]
		\centering
		\label{fig2.2a}
		\tikzset{every picture/.style={line width=0.75pt}} 
		\begin{tikzpicture}[x=0.75pt,y=0.65pt,yscale=-1,xscale=1 ]
			\draw  [color={rgb, 255:red, 2; green, 87; blue, 185 }  ,draw opacity=1 ] (154.08,77.13) -- (272.44,51.48) -- (215.9,315.75) -- (97.54,341.41) -- cycle ;
			\draw  [color={rgb, 255:red, 144; green, 19; blue, 254 }  ,draw opacity=1 ] (154.33,191.64) -- (233.13,174.46) -- (202.34,292.36) -- (123.54,309.54) -- cycle ;
			\draw  [color={rgb, 255:red, 208; green, 2; blue, 27 }  ,draw opacity=1 ] (160.55,217.99) -- (212.57,205.29) -- (192.08,279.83) -- (140.06,292.53) -- cycle ;
			\draw [color={rgb, 255:red, 10; green, 79; blue, 163 }  ,draw opacity=1 ]   (134.59,143.71) .. controls (112.24,139.69) and (39.49,121.44) .. (44.56,106.5) .. controls (49.63,91.56) and (166.53,109.99) .. (217.75,118.96) .. controls (268.97,127.94) and (356.03,152.48) .. (357.32,171.38) .. controls (358.59,190.09) and (254.27,197.35) .. (188.83,203.2) ;
			\draw [shift={(186.86,203.37)}, rotate = 354.85] [color={rgb, 255:red, 10; green, 79; blue, 163 }  ,draw opacity=1 ][line width=0.75]    (10.93,-4.9) .. controls (6.95,-2.3) and (3.31,-0.67) .. (0,0) .. controls (3.31,0.67) and (6.95,2.3) .. (10.93,4.9)   ;
			\draw [color={rgb, 255:red, 144; green, 19; blue, 254 }  ,draw opacity=1 ] [dash pattern={on 0.84pt off 2.51pt}]  (124.99,187.99) .. controls (62.56,164.8) and (19.34,152.99) .. (45.4,143.79) .. controls (71.45,134.59) and (312.79,183.83) .. (346.4,206.94) .. controls (379.84,229.93) and (312.22,236.47) .. (188.56,242.08) ;
			\draw [shift={(186.69,242.17)}, rotate = 357.42] [color={rgb, 255:red, 144; green, 19; blue, 254 }  ,draw opacity=1 ][line width=0.75]    (10.93,-4.9) .. controls (6.95,-2.3) and (3.31,-0.67) .. (0,0) .. controls (3.31,0.67) and (6.95,2.3) .. (10.93,4.9)   ;
			\draw [color={rgb, 255:red, 25; green, 100; blue, 188 }  ,draw opacity=1 ] [dash pattern={on 0.84pt off 2.51pt}]  (134.59,143.71) -- (181.41,153.79) ;
			\draw [color={rgb, 255:red, 144; green, 19; blue, 254 }  ,draw opacity=1 ] [dash pattern={on 0.84pt off 2.51pt}]  (124.99,187.99) -- (178.29,202.4) ;
			\draw  [line width=3] [line join = round][line cap = round] (187.04,154.8) .. controls (187.06,154.67) and (187.08,154.54) .. (187.1,154.41) ;
			\draw  [color={rgb, 255:red, 144; green, 19; blue, 254 }  ,draw opacity=1 ][line width=3] [line join = round][line cap = round] (180.76,201.87) .. controls (180.77,201.87) and (180.78,201.86) .. (180.79,201.86) ;
			\draw  [dash pattern={on 0.84pt off 2.51pt}]  (178.33,242) -- (184.48,262.32) ;
			\draw  [color={rgb, 255:red, 80; green, 227; blue, 194 }  ,draw opacity=1 ][line width=2.25] [line join = round][line cap = round] (178.63,242.67) .. controls (178.15,243.22) and (178.13,243.28) .. (178.46,242.78) ;
			\draw (183.37,137.32) node [anchor=north west][inner sep=0.75pt]    {$x$};
			\draw (167.22,186.32) node [anchor=north west][inner sep=0.75pt]  [color={rgb, 255:red, 144; green, 19; blue, 254 }  ,opacity=1 ,rotate=-345.54]  {$\phi _{t_{1}} (x)$};
			\draw (154.52,167.22) node [anchor=north west][inner sep=0.75pt]  [color={rgb, 255:red, 144; green, 19; blue, 254 }  ,opacity=1 ,rotate=-348.7]  {$\Phi _{\alpha }^{\mu} (H_{\delta}( x))$};
			\draw (153.07,223.47) node [anchor=north west][inner sep=0.75pt]  [color={rgb, 255:red, 208; green, 2; blue, 27 }  ,opacity=1 ,rotate=-351.28,xslant=-0.05]  {$(\Phi _{\alpha }^{\mu})^2(H_{\,\delta}( x))$};
			\draw (183.41,255.56) node [anchor=north west][inner sep=0.75pt]    {$y$};
			\draw (138.15,342.86) node [anchor=north west][inner sep=0.75pt]  [color={rgb, 255:red, 13; green, 81; blue, 162 }  ,opacity=1 ,rotate=-349.71]  {$H_{\delta}( x)$};
		\end{tikzpicture}\\
 \caption{Nested iterations of $H_{\delta}(x)$.}
	\end{figure}
This ensures that $\Phi^{\mu}_{\alpha}(y)=y$ and that $y$ is a periodic point with period at most $\mu(y)$. Moreover, since $y\in H_{\delta}(x)$, it follows that $y$ is stable. We conclude proving that $x$ is in the orbit of $y$. In fact, since $y$ is stable and $y \in \omega(x)$, there exist $r_1\in(0,r)$ and $t_2 > 0$ such that $$\phi_{t_2}(x) \in H_{r_1}(y) \subset T_\varepsilon^s(y)$$ and  Proposition \ref{3.09} ensures the existence of $\beta\in\h_{\varepsilon}(H_{r_1}(y))$ such that $$\lim_{s \rightarrow \infty}\diam(\Phi^s_{\beta}(H_{r_1}(y)))=0.$$ Since $x$ is recurrent, this ensures that $x$ and $y$ must be in the same orbit and, in particular, that $x$ is periodic.
\end{proof}
In Proposition \ref{stablesection}, we proved that if $x$ is stable and $y \in H_\delta(x)$, then $y$ is stable. In the following proposition, we prove that this can be generalized for any point $y$ sufficiently close to some iterate of $x$.

\begin{proposition}\label{prop5a}
If $(X, f)$ is a cw-expansive flow on a Peano continuum, $x\in X$ is stable, and $y \in \alpha(x)$, then $y$ is stable.
\end{proposition}
\begin{proof}
We assume that $x \in X$ is a stable point and prove that $\phi_{-t}(x)$ is also a stable point for every $t \geq 0$. Indeed, for $\varepsilon > 0$, choose $\delta > 0$ such that $H_\delta(x) \subseteq CT^s_\varepsilon(x)$ and consider $\delta_1\in(0,\delta)$ given by Proposition \ref{lema3.03} for $\delta$. Since $\diam(\Phi^s_{\alpha}(H_{\delta}(x)))\to0$ when $s\to+\infty$ (see Proposition \ref{3.09}) it follows from the choice of $\delta_1$ that
$$\Phi_{\beta}^t(H_{\delta_1}(\phi_{-t}(x)))\subset H_{\delta}(x)$$
for some $\beta\in\mathcal{H}(H_{\delta_1}(\phi_{-t}(x)))$. This ensures that $$H_{\delta_1}(\phi_{-t}(x))\subset T^s_{\eps}(\phi_{-t}(x))$$ and consequently $\phi_{-t}(x)$ is stable. Note that $\delta_1$ does not depend on $t$. To conclude that $y$ is stable it enough to consider $t>0$ such that $y\in H_{\delta_1}(\phi_{-t}(x))$ and use Proposition \ref{stablesection}.
\end{proof}

\begin{proposition}\label{Pro3.06}
If $(X, \phi)$ is a cw-expansive flow on a Peano continuum and $x \in X$ is a stable point, then $x$ is recurrent, therefore periodic.
\end{proposition}
\begin{proof}
We first prove that if $x$ is stable and $z \in \alpha(x)$, then $x \in \omega(z)$.
Proposition $\ref{prop5a}$ ensures that $z$ is stable. Thus, for each $\eta > 0$ there exists $\gamma > 0$ and $t>0$ such that $$\phi_{-t}(x) \in H_{\gamma}(z) \subset T^{s}_{\eta}(z).$$ Then there exists $\alpha \in Rep(\mathbb{R}^+)$ such that
$$d(\phi_{\alpha(s)}(\phi_{-t}(x)),\phi_s(z))\leq\eps \,\,\,\,\,\, \text{for every} \,\,\,\,\,\, s\geq0.$$
Choosing $s>0$ such that $\alpha(s)=t$, the last inequality becomes
$$d(x,\phi_s(z))\leq \eta.$$ Since this can be done for each $\eta>0$, it follows that $x\in \omega(z)$. Now, since $x$ is stable, there exists $\delta > 0$ such that $H_\delta(x) \subseteq CT^s_\varepsilon(x)$ and Proposition \ref{3.09} ensures that
$$d(\Phi^t_{\alpha}(z),\phi_t(x))\to0 \,\,\,\,\,\, \text{when} \,\,\,\,\,\, t\to+\infty$$ for some $\alpha\in\mathcal{H}(H_{\delta}(x))$. This and $x\in \omega(z)$ ensure that $x\in\omega(x)$, i.e. $x$ is recurrent, while Proposition \ref{reccurentperiodic} ensures that $x$ is periodic.
\end{proof}

Before stating the next result some remarks are in order. 
Consider first the simpler case of homeomorphisms. 
From the definitions, if $X$ is finite (for instance, a single point) then each point is stable (for any homeomorphism, in this case, permutation, of $X$). These examples are also expansive. 
If $X$ is a single point then we have a Peano continuum with an expansive homeomorphism that presents stable points. Thus, to be precise, in the next result we have to exclude the case of $X$ being a circle.

\begin{theorem}\label{NPE3.07}
    Let $(X,\phi_t)$ be a cw-expansive flow defined on a Peano continuum $X$ which is not a circle, then there are no stable points for $\phi_t$ and $\phi_{-t}$.
\end{theorem}
\begin{proof}
By contradiction, suppose that $x \in X$ is a stable point and let $\delta>0$ be such that $H_\delta(x) \subset T^s_\varepsilon(x)$. We can assume that $\delta$ also satisfies the conclusion of Proposition \ref{stablesection}, reducing $\delta$ if necessary. Thus, if $z\in H_{\delta}(x)$, then $z$ is stable and
by Proposition \ref{Pro3.06}, periodic. 
By Proposition \ref{3.09} we know that $x$ and $z$ are asymptotic, but being periodic they must be the same point, $z=x$. We conclude that
$H_{\delta}(x)=\{x\}$. This means that the orbit of the periodic point $x$ is an open subset. 
Since $X$ is connected we conclude that $X$ is the orbit of $x$, \textit{i.e.} $X$ is a circle. 
This contradiction proves the result.
\end{proof}

\begin{theorem}
If $(X,\phi_t)$ is a cw-expansive flow defined on a Peano continuum and $\varepsilon > 0$ is a cw-expansivity constant, then there exists $\delta > 0$ such that for each $x \in X$ there exist continua $C_x\in\mathcal{T}^{s}_\varepsilon$ and $D_x\in\mathcal{T}^{u}_\varepsilon$ such that $x \in C_x \cap D_x$ and $\text{diam}(C_x) = \text{diam}(D_x) = \delta$.
\end{theorem}
\begin{proof}
Proposition $\ref{lema3.03}$ ensures the existence of $\delta\in(0,\eps)$ corresponding to $\varepsilon/2$. Let $x \in X$ and $z \in \omega(x)$, then there exist $k_0>0$ and a sequence $t_k \rightarrow \infty$ as $k \rightarrow \infty$ such that $x_k = \phi_{t_k}(x) \in H_{\frac{\delta}{2}}(z)$ for every $k \geq k_0$. According to Theorem $\ref{NPE3.07}$, the point $z$ is not stable for $\phi_{-t}$. Therefore, there exists $t_0 > 0$ and an $\alpha \in \h(H_{\frac{\delta}{2}}(z))$ such that
$$\d(\Phi^{-t_0}_\alpha(z,H_{\frac{\delta}{2}}(z)))>\varepsilon.$$
We can assume that $t_k > t_0$ for all $k \geq 1$. 
By the continuity of the field of transversal sections, given $0<\epsilon'<\epsilon$ we have that 
for large enough values of $k$ the next inequality holds
$$\d (\Phi^{-t_0}_\alpha(x_k,H_\delta(x_k)))>\varepsilon/2.$$
Then consider $C_k \subset H_\delta(x_k)$ sub-continua such that $x_k \in C_k$, 
$$\text{diam} (\Phi^{-t}_{\alpha}(x_k, C_k)) \leq \varepsilon \,\,\,\,\,\, \text{for every} \,\,\,\,\,\, t \in [0, t_k] \,\,\,\,\,\, \text{and}$$ 
$$\text{diam} (\Phi^{-s_k}_\alpha(x_k, C_k)) = \varepsilon/2 \,\,\,\,\,\, \text{for some} \,\,\,\,\,\, s_k \in [0, t_k].$$
Thus, according to Lemma $\ref{lema3.03}$, we would have
    $$\d (\Phi^{-t}_\alpha(x_k, C_k))\geq \delta \,\,\,\,\,\, \text{for every} \,\,\,\,\,\, t\geq s_k.$$
    In particular, we have
    \begin{itemize}
		\item $\d (\Phi^{-t_k}_\alpha(x_k, C_k))\geq \delta$,\\
		\item $x \in D_k=\Phi^{-t_k}_\alpha(x_k, C_k)$,\\
		\item $\d (\Phi^{t}_\alpha(x, D_k))\leq \varepsilon$ for every $t\in [0, t_k]$.\\
	\end{itemize}
 Therefore, the set $D = \lim_{k \rightarrow \infty} D_k$ is a continuum, $\text{diam} \ D \geq \delta$, $x \in D$, and $D \in T^{s}_\varepsilon$. We can proceed similarly for the unstable case and conclude the proof.
\end{proof}

\hspace{-0.45cm}\textbf{Acknowledgments.}
Alfonso Artigue was supportes by ANII
and PEDECIBA at Uruguay.
Bernardo Carvalho was supported by Progetto di Eccellenza MatMod@TOV grant number PRIN 2017S35EHN and Margoth Tacuri was also supported by Fapemig grant number APQ-00036-22. This article is part of the Ph.D. thesis of Margoth Tacuri, under the supervision of the other authors, defended at the Federal University of Minas Gerais in 2023.

\vspace{1.5cm}
\noindent

{\em A. Artigue}
\vspace{0.2cm}

\noindent

Departamento de Matem\'atica y Estadística del Litoral,

Universidad de la República,

Gral. Rivera 1350, Salto, Uruguay
\vspace{0.2cm}

\email{artigue@unorte.edu.uy}

\vspace{1.5cm}
\noindent

{\em B. Carvalho}
\vspace{0.2cm}

\noindent

Dipartimento di Matematica,

Università degli Studi di Roma Tor Vergata

Via Cracovia n.50 - 00133

Roma - RM, Italy
\vspace{0.2cm}

\email{carvalho@mat.uniroma2.it}

\vspace{1.5cm}
\noindent

{\em M. Tacuri}
\vspace{0.2cm}

\noindent

Universidade Nacional de Ingeniería - UNI

Av. Túpac Amaru 210, Rímac 15333

Lima, Perú.
\vspace{0.2cm}

simon.tacuri.m@uni.edu.pe

\end{document}